\documentclass[10pt]{amsart}
\usepackage[margin=1.45in]{geometry}
\newtheorem{theorem}{Theorem}[section]
\newtheorem{lemma}[theorem]{Lemma}
\newtheorem{proposition}{Proposition}[section]

\theoremstyle{definition}
\newtheorem{definition}[theorem]{Definition}

\theoremstyle{remark}
\newtheorem{remark}[theorem]{Remark}

\numberwithin{equation}{section}

\begin{document}

\title[polyatomic BGK model]{Cauchy problem for the ellipsoidal BGK model for polyatomic particles}

\author{Sa Jun Park, SEOK-BAE YUN }
\address{Department of Mathematics, Sungkyunkwan University, Suwon 440-746, Republic of Korea}
\email{parksajune@skku.edu}
\address{Department of Mathematics, Sungkyunkwan University, Suwon 440-746, Republic of Korea}
\email{sbyun01@skku.edu}



\begin{abstract}
We establish the existence and uniqueness of mild solutions for the polyatomic ellipsoidal BGK model, which is a relaxation type kinetic model describing the evolution of polyatomic gaseous system at the mesoscopic level.
\end{abstract}

\maketitle
\section{Introduction}
The derivation of the celebrated Boltzmann equations relies heavily on the assumption that the gas consists of monatomic particles,
which is not the case for most of the realistic gases. Efforts to derive Boltzmann type kinetic models soon confront with the difficulty that it is virtually impossible to write  the pre - and post - collision velocities in an explicit form, since polyatomic molecules can possess arbitrarily complicated structures.
In search of tractable model equation for polyatomic gases that avoids such difficulties, a BGK type model was suggested
as a generalization of the ellipsoidal BGK model \cite{ALPP,BL,Brun,BS2,PL,Shen}:
\begin{align}\label{ESBGK}
\begin{split}
\partial_t f+v\cdot \nabla_x f &= A_{\nu,\theta} (\mathcal{M}_{\nu,\theta}(f)-f) \cr
f(0,x,v,I)&= f_0(x,v,I).
\end{split}
\end{align}
Unlike the monatomic case, a new variable $I$ related to the internal energy due to the rotational and vibrational motions of the molecules is introduced so that the velocity distribution function $f(t,x,v,I)$  represents the number density  on $(x,v)\in\mathbb{T}_x^3 \times \mathbb{R}_v^3$ at time $t$ with internal energy $I^{2/\delta}\in\mathbb{R}^+$, where $\delta$ is the number
of degrees of freedom except for the translational motion.
We consider the fixed collision frequency $A_{\nu,\theta}=1/(1-\nu+\nu\theta)$ throughout this paper.
Two relaxation parameters $-1/2 < \nu <1$ and $0\leq \theta \leq 1$ are chosen in such a way that Prandtl number 
 and the second viscosity coefficient
 computed through the Chapmann-Enskog expansion, agrees  with the physical data.
(See \cite{ABLP,BS,Cai,Shen}).\newline
The polyatomic Gaussian $\mathcal{M}_{\nu,\theta}(f)$ reads
\begin{align}
\mathcal{M}_{\nu,\theta}(f) = \frac{\rho\Lambda_{\delta}}{
\sqrt{\det(2\pi \mathcal{T}_{\nu,\theta})}(T_\theta)^\frac{\delta}{2}}\exp\left(-\frac{1}{2}(v-U)^{\top}\mathcal{T}_{\nu,\theta}^{-1}(v-U)-\frac{I^{\frac{2}{\delta}}}{T_{\theta}}\right)
\end{align}
with normalizing factor
\begin{align*}
\Lambda_{\delta}^{-1} = \int_{\mathbb{R}^+} \exp(-I^{\frac{2}{\delta}})dI.	
\end{align*}
The macroscopic local density $\rho(t,x)$, bulk velocity $U(t,x)$, stress tensor $\Theta(t,x)$ and internal energy $E_{\delta}(t,x)$ are defined respectively by
\begin{align}
\begin{split}
\rho(t,x) &= \int_{\mathbb{R}^3 \times \mathbb{R}^+} f(t,x,v,I) dvdI \cr
U(t,x) &= \frac{1}{\rho}\int_{\mathbb{R}^3 \times \mathbb{R}^+} vf(t,x,v,I) dvdI \cr
\Theta(t,x) &= \frac{1}{\rho}\int_{\mathbb{R}^3 \times \mathbb{R}^+} f(t,x,v,I)\big(v-U(t,x)\big)\otimes\big(v-U(t,x)\big)dvdI \cr
E_{\delta}(t,x) &=\int_{\mathbb{R}^3 \times \mathbb{R}^+} \left(\frac{1}{2}|v-U(t,x)|^2+I^{\frac{2}{\delta}}\right)f(t,x,v,I) dvdI.
\end{split}
\end{align}
 We split the internal energy $E_{\delta}$ into
the internal energy from the translational motion  $E_{tr}$ and the one from the non-translational motion $E_{I,\delta}$ :
\begin{align*}
E_{tr} &=\int_{\mathbb{R}^3 \times \mathbb{R}^+} \frac{1}{2}|v-U|^2f dvdI, \cr
E_{I,\delta} &=\int_{\mathbb{R}^3 \times \mathbb{R}^+} I^{\frac{2}{\delta}}f dvdI,
\end{align*}
and define the corresponding temperatures $T_{\delta}, T_{tr}$ and $T_{I,\delta}$ by the equi-partition principle:
\begin{align*}
E_{\delta}=\frac{3+\delta}{2}\rho T_{\delta},\quad
E_{tr}=\frac{3}{2}\rho T_{tr},\quad
E_{I,\delta}=\frac{\delta}{2}\rho T_{I,\delta}.
\end{align*}
Note that  $T_{\delta}$ is a convex combination of $T_{tr}$ and $T_{I,\delta}$:
\begin{align}\label{convex}
T_{\delta}=\frac{3}{3+\delta}T_{tr} + \frac{\delta}{3+\delta}T_{I,\delta}.
\end{align}
Then, the relaxation temperature $T_{\theta}$ and the corrected temperature tensor $\mathcal{T}_{\nu,\theta}$
are defined as follows:
\begin{align}
\begin{split}
T_{\theta} &= \theta T_{\delta}+(1-\theta)T_{I,\delta}, \cr
\mathcal{T}_{\nu,\theta} &= \theta T_{\delta}Id+(1-\theta)\left\{(1-\nu)T_{tr}Id+\nu\Theta\right\}.
\end{split}
\end{align}

The relaxation operator satisfies the following cancellation properties:
\begin{align}\label{cancellation}
\begin{split}
&\int_{\mathbb{T}^3 \times \mathbb{R}^3 \times \mathbb{R}^+}
(\mathcal{M}_{\nu,\theta}(f)-f)dxdvdI = 0 \cr
&\int_{\mathbb{T}^3 \times \mathbb{R}^3 \times \mathbb{R}^+}
v(\mathcal{M}_{\nu,\theta}(f)-f)dxdvdI = 0\cr
&\int_{\mathbb{T}^3 \times \mathbb{R}^3 \times \mathbb{R}^+}
\left(\frac{1}{2}|v|^2+I^{\frac{2}{\delta}}\right)(\mathcal{M}_{\nu,\theta}(f)-f)dxdvdI = 0,
\end{split}
\end{align}
yielding the conservation of mass, momentum and energy respectively.
The entropy dissipation for the polyatomic gas was proved by Andries and Perthame et al \cite{ALPP}. (See also
\cite{BS2,PY2})
\[
\frac{d}{dt}\int_{\mathbb{T}^{3}\times\mathbb{R}^{3}\times \mathbb{R}^+} f(t)\ln f(t) dxdvdI\leq 0.
\]

%
%
%
%
\section{main result}
\begin{definition}
Let $T>0$. $f \in C_+([0,T];\|\cdot\|_{L_q^{\infty}})$ is said to be a mild solution for (\ref{ESBGK}) if it satisfies
\begin{align*} f(t,x,v,I)=e^{-A_{\nu,\theta}t}f_0(x-vt,v,I)+A_{\nu,\theta}\int_0^te^{-A_{\nu,\theta}(t-s)}\mathcal{M}_{\nu,\theta}(f)(x-(t-s)v,v,s,I)ds,
\end{align*}
\end{definition}
\noindent where the weighted norm $\|\cdot\|_{L^{\infty}_q}$ is defined by
\begin{align*}
\|f(t)\|_{L_q^{\infty}}=ess\sup_{x,v,I} |f(t,x,v,I)(1+|v|^2+I^{\frac{2}{\delta}})^\frac{q}{2}|.
\end{align*}
Our main result is as follows:
\begin{theorem}\label{MainThm}
Let $0<\theta\leq 1$, $-1/2<\nu<1$, $\delta>0$ and $q>5+\delta$. Suppose there exist positive constants $C_u, C_l$ and $C_1$ such that
\begin{align*}
 \|f_0\|_{L_q^{\infty}} < C_u,\quad\int_{\mathbb{R}^3\times\mathbb{R}^+}f_0(x-vt,v,I)dvdI\geq C_l >0.
\end{align*}
Then, for any final time $T>0$, there exists a unique mild solution $f\in C_+([0,T];\|\cdot\|_{L_q^{\infty}})$  for (\ref{ESBGK}) such that
\begin{enumerate}
\item $f$ is bounded on $t\in[0,T)$ as
\begin{align*}
	\|f(t)\|_{L_q^{\infty}} \leq e^{C_1t}\|f_0\|_{L_q^{\infty}}.
\end{align*}
\item There exist positive constants $C_{T,f_0}$, $C_{T,f_0,\delta}$ and $C_{T,f_0,\delta,q}$ such that
\begin{align*}
&\rho(x,t) \geq C_{T,f_0},\cr
&T_{\delta}(x,t)\geq C_{T,f_0,\delta},\cr
&\rho(x,t)+|U(x,t)|+T_{\delta}(x,t)\leq C_{T,f_0,\delta,q}.
\end{align*}
\item Conservation laws of mass, momentum and energy hold:
\begin{align*}
\frac{d}{dt}\int_{\mathbb{T}^3\times\mathbb{R}^3\times\mathbb{R}^+}f\left(1,v,\frac{1}{2}|v|^2+I^{\frac{2}{\delta}}\right)dxdvdI=0.
\end{align*}
\item H-theorem holds:
\begin{align*}
\frac{d}{dt}\int_{\mathbb{T}^3\times\mathbb{R}^3\times\mathbb{R}^+}f\ln fdxdvdI =
\int_{\mathbb{T}^3\times\mathbb{R}^3\times\mathbb{R}^+}
\big(\mathcal{M}_{\nu,\theta}(f)-f\big)\ln f dxdvdI\leq 0.
\end{align*}
\end{enumerate}
\end{theorem}
\begin{remark}
When $\theta=0$, all the above estimates break down. Therefore, this case should be considered separately. See Section 7 for the discussion of this case.
\end{remark}

Ever since it was introduced in \cite{BGK,Wel}, the BGK model has seen huge applications in engineering and physics. The first mathematical
study was carried out by Perthame in \cite{Perthame}, where  the existence of weak solutions was proven under the assumption of finite mass, momentum, energy and entropy. Perthame and Pulvirenti \cite{PP} then considered the class of solution space in which the uniqueness is guaranteed. It was later extended to the whole space
\cite{Mischler}, and to $L^p$ solutions \cite{ZH}. The Cauchy problem in the presence of external force or mean field was considered in \cite{BC,WZ,Zhang}. Ukai studied a stationary problem on a bounded interval in \cite{Ukai-BGK}. The existence and asymptotic behavior near a global maxwellian were studied in \cite{Bello,Yun,Zhang}. For various macroscopic limits of BGK type models, see \cite{DMOS,LT,MMM,Mellet,SR1,SR2}. Recently, Holway's ellipsoidal generalization of the original BGK model (ES-BGK model) was re-suggested in \cite{ALPP} with the first proof of $H$-theorem, and studied analytically in a series of paper \cite{BS,BY,DWY,PY1,Yun2,Yun3,Yun33,Yun4}.
Mathematical study on the polyatomic BGK model is in its initial state. See \cite{BS2} for the derivation of this model.
In \cite{PY2}, the entropy -entropy production estimate was derived.
\cite{Yun} studies the existence in the near-global-polyatomic Maxwellian-regime. A dichotomy in the dissipative estimate was also observed.

For the numerical results of BGK model - monatomic, or polyatomic - we refer to \cite{ABLP,Cai,FJ,GT,GRS,Issautier,PPuppo,RSY,Z-Stru} and references therein. A nice survey on various mathematical and physical issues on kinetic equations can be found in
\cite{C,CIP,CC,GL,PL,Sone,Sone2,UT,V}.
\newline

Following is the notational convention kept throughout this paper:
\begin{itemize}
\item Constants, usually denoted by $C$, are defined generically. Their value may vary line by line but can be
computed in principle.
%
\item When necessary, we use $C_{a,b,c,..}$, to show the dependence, not necessarily exclusive, on $a,b,c\cdots$.
\item For $\kappa\in\mathbb{R}^3$, $\kappa^{\top}$ denotes its transpose.
\item For symmetric $n\times n$ matrices $A$ and $B$, $A\leq B$ means $B-A$ is positive definite. That is, $k^{\top}\big\{B-A\big\}k\geq0$ for all $k\in \mathbb{R}^n$.

\end{itemize}

The paper is organized as follows: In the following Section 3, we establish several estimates for macroscopic variables.
In Section 4, we define our solution space and show that the approximate solutions lie in that space for all steps of iterations.
Section 5 is devoted to showing that the relaxation operator is Lipschitz continuous in the solution space.
In Section 6, we combine all the previous results to complete the existence proof.
The reason why the case $\theta=0$ should be treated independently is briefly  discussed in Section 7.
In the appendix, we prove the cancellation property of the relaxation operator.

\section{Estimates on macroscopic fields}
\begin{lemma}\label{Temperature}
Let $\delta>0$, $-1/2<\nu<1$ and $0<\theta\leq 1$. Suppose  $\rho>0$, $T_{tr}>0$ and $T_{I,\delta}>0$. Then temperature tensor $\mathcal{T}_{\nu,\theta}$ and
the relaxation temperature $T_{\theta}$ satisfy the following equivalence type  estimates:
\begin{align*}
&(1)\ \theta T_{\delta}Id \leq \mathcal{T}_{\nu,\theta} \leq \frac{1}{3}C_{\nu}\big\{3+\delta(1-\theta)\big\} T_{\delta}Id, \cr
&(2)\ \theta T_{\delta} \leq T_{\theta} \leq \frac{1}{\delta}\big\{\delta+3(1-\theta)\big\} T_{\delta},
\end{align*}
where $C_{\nu}=\max_{\nu}\{1-\nu,1+2\nu\}$.
\end{lemma}
\begin{proof}
(1) \emph{\bf (a) Upper bound}: Recalling the definition of $\mathcal{T}_{\nu,\theta}$, we write
\begin{align*}
\rho\mathcal{T}_{\nu,\theta}&=\theta \rho T_{\delta}Id+(1-\theta)\left\{(1-\nu)\rho T_{tr}Id+\nu \rho \Theta\right\}\cr
&=\theta\rho T_{\delta}Id+(1-\theta)\bigg\{(1-\nu)\rho T_{tr}Id+\nu \int_{\mathbb{R}^3\times\mathbb{R}^+} f (v-U)\otimes(v-U) dvdI\bigg\}.
\end{align*}
From the identity
\[
k^{\top}\big\{(v-U)\otimes (v-U)\big\}k=\big\{(v-U)\cdot k\big\}^2,\quad \mbox{for }k \in \mathbb{R}^3,
\]
we derive
\begin{align}\label{that}
k^{\top}\{\rho\mathcal{T}_{\nu,\theta}\}k
&=\theta\rho T_{\delta}|k|^2 +
(1-\theta)\left\{(1-\nu)\rho T_{tr}|k|^2+\nu \int_{\mathbb{R}^3\times\mathbb{R}^+} f \big\{(v-U)\cdot k\big\}^2 dvdI\right\}.
\end{align}
If $0\leq\nu<1$, using the Cauchy-Schwartz inequality, we get
\begin{align*}
\int_{\mathbb{R}^3\times\mathbb{R}^+} f \big\{(v-U)\cdot k\big\}^2 dvdI
\leq\int_{\mathbb{R}^3\times\mathbb{R}^+} f |v-U|^2|k|^2 dvdI =3\rho T_{tr}|k|^2,
\end{align*}
so that
\begin{align*}
k^{\top}\{\rho\mathcal{T}_{\nu,\theta}\}k
&\leq \theta\rho T_{\delta}|k|^2+(1-\theta)\left\{(1-\nu)\rho T_{tr}|k|^2+3\nu\rho T_{tr}|k|^2\right\} \cr
&=\theta\rho T_{\delta}|k|^2+ (1-\theta)(1+2\nu)\rho T_{tr}|k|^2\cr
&\leq(1+2\nu)\rho\left\{\theta T_{\delta}+ (1-\theta) T_{tr}\right\}|k|^2.
\end{align*}
In the case of $-1/2 < \nu < 0$, the last term in (\ref{that}) is non-positive. Thus
\begin{align*}
k^{\top}\{\rho\mathcal{T}_{\nu,\theta}\}k
&\leq \theta\rho T_{\delta}|k|^2+(1-\theta)(1-\nu)\rho T_{tr}|k|^2\cr
&\leq (1-\nu)\rho\left\{\theta T_{\delta}|k|^2+(1-\theta) T_{tr}\right\}|k|^2
\end{align*}
Combining these two cases, we arrive at
\begin{align}\label{from}
k^{\top}\left\{\rho\mathcal{T}_{\nu,\theta}\right\}k
\leq  \max\{1-\nu,1+2\nu\} \rho \left\{(1-\theta)T_{tr}+\theta T_{\delta}\right\}|k|^2.
\end{align}

Now, we recall (\ref{convex}) to see
\begin{equation}\label{derive}
T_{\delta}= \frac{3}{3+\delta}T_{tr}+\frac{\delta}{3+\delta}T_{I,\delta}\geq \frac{3}{3+\delta}T_{tr},
\end{equation}
or
\[
T_{tr}\leq \frac{3+\delta}{3}T_{\delta}
\]
to derive from (\ref{from}) that
\[
k^{\top}\left\{\rho\mathcal{T}_{\nu,\theta}\right\}k\leq  \frac{1}{3}\max\{1-\nu,1+2\nu\}\rho\big\{3+\delta(1-\theta)\big\}T_{\delta}|k|^2.
\]
This implies the desired estimate, since we assumed $\rho>0$.\newline
\noindent{\bf (b) Lower bound:} Denote the last term in (\ref{that}) by $A$:
\[
A=(1-\nu)\rho T_{tr}|k|^2+\nu \int_{\mathbb{R}^3\times\mathbb{R}^+} f \big\{(v-U)\cdot k\big\}^2 dvdI.
\]
Then, when $0<\nu<1$, $A$ satisfies
\begin{align*}
A\geq (1-\nu)\rho T_{tr}|k|^2,
\end{align*}
whereas we have
\begin{align*}
A\geq(1-\nu)\rho T_{tr}|k|^2+\nu\left(\int_{\mathbb{R}^3}f|v-U|^2dv\right)|k|^2
=(1+2\nu)\rho T_{tr}|k|^2,
\end{align*}
for $-1/2<\nu\leq0$. Therefore, we conclude from our assumption on $\rho$ and $T_{tr}$ that $A\geq0$.
Thus, we deduce from (\ref{that})
\begin{align}\label{from2}
k^{\top}\left\{\rho\mathcal{T}_{\nu,\theta}\right\}k
\geq  \theta\rho T_{\delta}|k|^2+(1-\theta)A\geq\theta \rho T_{\delta}|k|^2,
\end{align}
which gives the desired result.
\newline

\noindent (2) From the definition of $T_{\delta}$ (\ref{convex}), we have
\begin{equation}\label{derive2}
T_{\delta}= \frac{3}{3+\delta}T_{tr}+\frac{\delta}{3+\delta}T_{I,\delta}\geq \frac{\delta}{3+\delta}T_{I,\delta},
\end{equation}
so that
\[
T_{I,\delta}\leq \frac{3+\delta}{\delta}T_{\delta}.
\]
Therefore,
\begin{align*}
T_{\theta}&= (1-\theta)T_{I,\delta}+\theta T_{\delta}\cr
&\leq (1-\theta)\bigg(\frac{3+\delta}{\delta}T_{\delta}\bigg)+\theta T_{\delta}\cr
&=\frac{1}{\delta}\left\{\delta+3(1-\theta)\right\} T_{\delta}.
\end{align*}
The lower bound comes directly from the definition:
\begin{equation*}
T_{\theta}= (1-\theta)T_{I,\delta}+\theta T_{\delta}\geq \theta T_{\delta}.
\end{equation*}
\end{proof}
\begin{lemma}\label{Lemma1} Assume $\rho>0$ and $\|f\|_{L_q^{\infty}}<\infty$. Then we have
\begin{align*}
\rho\leq C_{\delta}\|f\|_{L_q^{\infty}}T_{\delta}^{\frac{3+\delta}{2}}
\end{align*}
for
\[
C_{\delta}=2^{\frac{7}{2}}\pi^2(3+\delta)^{\frac{1+\delta}{2}}\delta.
\]
\end{lemma}
\begin{proof}
We divide the integral domain as
\begin{align}\label{Lemma11}
\begin{split}
\rho &= \int_{\mathbb{R}^3 \times \mathbb{R}^+}f dvdI \cr
&\leq \int_{\frac{1}{3+\delta}|v-U|^2+\frac{2}{3+\delta}I^{\frac{2}{\delta}}>R^2}fdvdI+  \int_{\frac{1}{3+\delta}|v-U|^2+\frac{2}{3+\delta}I^{\frac{2}{\delta}}\leq R^2}fdvdI \cr
&\equiv I_1+I_2.
\end{split}
\end{align}
From the definition of $T_{\delta}$, we see that
\begin{align*}
I_1 &\leq \frac{1}{R^2}\int_{\frac{1}{3+\delta}|v-U|^2+\frac{2}{3+\delta}I^{\frac{2}{\delta}}>R^2}\left(\frac{1}{3+\delta}|v-U|^2+\frac{2}{3+\delta}I^{\frac{2}{\delta}}\right)fdvdI\cr
&\leq \frac{1}{R^2}\rho T_{\delta}.
\end{align*}
For $I_2$, we estimate
\begin{align*}
I_2&\leq \left(\int_{\frac{1}{3+\delta}|v-U|^2+\frac{2}{3+\delta}I^{\frac{2}{\delta}}\leq R^2}dvdI\right)\|f\|_{L_{q}^{\infty}},
\end{align*}
and make a change of variable:
\begin{align*}
\sqrt{\frac{1}{3+\delta}}(v_1-U_1)&=r\sin\varphi\cos\theta\sin k,\cr
\sqrt{\frac{1}{3+\delta}}(v_2-U_2)&=r\sin\varphi\sin\theta\sin k, \cr
\sqrt{\frac{1}{3+\delta}}(v_3-U_3)&=r\cos\varphi\sin k,\cr
\sqrt{\frac{2}{3+\delta}}I^{\frac{1}{\delta}} &= r\cos k,
\end{align*}
for $0\leq r\leq R, 0\leq\varphi\leq\pi,  0\leq\theta\leq2\pi,  0\leq k \leq \frac{\pi}{2}$.
The the Jacobian
\[
J=\frac{\partial(v_1,v_2,v_3,I)}{\partial(r,\varphi,\theta,k)},
\]
is computed as
\begin{align*}
|J|&=\left(3+\delta\right)^{\frac{3}{2}}\left(\frac{3+\delta}{2}\right)^{\frac{\delta}{2}} \cr
&\times
\left|\det\left(\begin{array}{cccc}
			\cos\theta\sin\varphi\sin k&
			r\cos\theta\sin k \cos\varphi&
			-r\sin\varphi\sin\theta\sin k&
			r\sin\varphi\cos\theta\cos k
			\cr
			\sin\varphi\sin\theta\sin k&
			r\cos\varphi\sin\theta\sin k&
			r\sin\varphi\cos\theta\sin k&
			r\sin\varphi\sin\theta\cos k
			\cr
			\cos\varphi\sin k &
			-r\sin\varphi\sin k&
			0&
			r\cos\varphi\cos k
			\cr
			\delta r^{\delta-1}\cos^{\delta}k&
			0&
			0&
			\delta r^{\delta}\cos^{\delta-1}k \sin k
\end{array}
\right)\right|\cr
&=\delta\left(3+\delta\right)^{\frac{3}{2}}\left(\frac{3+\delta}{2}\right)^{\frac{\delta}{2}} r^{\delta+2}|\sin\varphi\cos^{\delta-1}k\sin^2k|.
\end{align*}
so that
\begin{align*}
I_2&\leq\|f\|_{L_{q}^{\infty}}
\int_{0}^{\frac{\pi}{2}}\int_{0}^{\pi}\int_{0}^{2\pi}\int_{0}^{R} \delta(3+\delta)^{\frac{3}{2}}\left(\frac{3+\delta}{2}\right)^{\frac{\delta}{2}}r^{\delta+2}|\sin\varphi\cos^{\delta-1}k\sin^2k| drd\theta d\varphi dk  \cr
&\leq \|f\|_{L_{q}^{\infty}}
\left\{(3+\delta)^{\frac{3}{2}}\left(\frac{3+\delta}{2}\right)^{\frac{\delta}{2}}\frac{2\pi^2\delta}{3+\delta}      \right\}R^{3+\delta} \cr
&= \|f\|_{L_{q}^{\infty}}\left\{2^{\frac{2-\delta}{2}}\pi^2(3+\delta)^{\frac{1+\delta}{2}}\delta \right\}R^{3+\delta} .
\end{align*}
Thus, $(\ref{Lemma11})$ can be estimated as follows:
\begin{align*}
\rho \leq
\frac{1}{R^2}\rho T_{\delta}+\left\{2^{\frac{2-\delta}{2}}\pi^2(3+\delta)^{\frac{1+\delta}{2}}\delta \right\}R^{3+\delta}\|f\|_{L_q^{\infty}}.
\end{align*}
We optimize this by setting
\[
R^{5+\delta}=\frac{\rho T_{\delta}}{\left\{2^{\frac{2-\delta}{2}}\pi^2(3+\delta)^{\frac{1+\delta}{2}}\delta \right\}\|f\|_{L_q^{\infty}}}
\]
to get
\begin{align*}
\rho
&\leq
2\left\{2^{\frac{2-\delta}{2}}\pi^2(3+\delta)^{\frac{1+\delta}{2}}\delta \right\}^{\frac{2}{5+\delta}}\left\{\rho T_{\delta}\right\}^{\frac{3+\delta}{5+\delta}},
\end{align*}
which implies
\begin{align*}
\rho
\leq
\left\{2^{\frac{7}{2}}\pi^2(3+\delta)^{\frac{1+\delta}{2}}\delta\right\}\|f\|_{L_q^{\infty}}T_{\delta}^{\frac{3+\delta}{2}}.
\end{align*}
This completes the proof.
\end{proof}
\begin{lemma}\label{Lemma2} Assume $\rho>0$ and $\|f\|_{L_q^{\infty}}>0$. Then, for $q>5+\delta$, we have
\begin{align*}
\rho(T_\delta+|U|^2)^{\frac{q-\delta-3}{2}} \leq C_{\delta,q}\|f\|_{L_{q}^{\infty}},
\end{align*}
where constant $C_{\delta,q}$ is given by
\[
C_{\delta,q}=\left\{\frac{2^{\frac{q-2\delta-1}{2}}\pi^2(3+\delta)^{\frac{q}{2}}\delta}{q-\delta-5}\right\}.
\]
\end{lemma}
\begin{proof}
From the definition of $T_{\delta}$, we write
\begin{align*}
\rho\left(T_{\delta}+\frac{1}{3+\delta}|U|^2\right)
&= \int_{\mathbb{R}^3\times\mathbb{R}^+} \left(\frac{1}{3+\delta}|v|^2+\frac{2}{3+\delta}I^{\frac{2}{\delta}}\right)f dvdI.
\end{align*}
We then split the integral into the following two part as
\begin{align}\label{Lemma22}
\begin{split}
\rho\left(T_{\delta}+\frac{1}{3+\delta}|U|^2\right)
&= \int_{\frac{1}{3+\delta}|v|^2+\frac{2}{3+\delta}I^{\frac{2}{\delta}}>R^2} \left(\frac{1}{3+\delta}|v|^2+\frac{2}{3+\delta}I^{\frac{2}{\delta}}\right)f dvdI \cr
&+ \int_{\frac{1}{3+\delta}|v|^2+\frac{2}{3+\delta}I^{\frac{2}{\delta}}\leq R^2} \left(\frac{1}{3+\delta}|v|^2+\frac{2}{3+\delta}I^{\frac{2}{\delta}}\right)f dvdI\cr
&=I_1+I_2.
\end{split}
\end{align}
The estimate for $I_2$ is simple:
\begin{align*}
I_2\leq R^2\int_{\frac{1}{3+\delta}|v|^2+\frac{2}{3+\delta}I^{\frac{2}{\delta}}\leq R^2} f dvdI\leq R^2\rho.
\end{align*}
For $I_1$, we  extract $\|f\|_{L^{\infty}_q}$ out of the integral:
\begin{align*}
I_1
&\leq \int_{\frac{1}{3+\delta}|v|^2+\frac{2}{3+\delta}I^{\frac{2}{\delta}}>R^2}
\frac{\left(\frac{1}{3+\delta}|v|^2+\frac{2}{3+\delta}I^{\frac{2}{\delta}}\right)^{\frac{q}{2}}f}{\left(\frac{1}{3+\delta}|v|^2+\frac{2}{3+\delta}I^{\frac{2}{\delta}}\right)^{\frac{q-2}{2}}} ~dvdI\cr
&\leq\|f\|_{L_q^{\infty}}
\int_{\frac{1}{3+\delta}|v|^2+\frac{2}{3+\delta}I^{\frac{2}{\delta}}>R^2}\frac{1}{\left(\frac{1}{3+\delta}|v|^2+\frac{2}{3+\delta}I^{\frac{2}{\delta}}\right)^{\frac{q-2}{2}}}dvdI,
\end{align*}
and use the same change of variable as in the proof of the previous lemma to estimate
\begin{align*}
I_1&\leq\|f\|_{L_q^{\infty}} \int_{0}^{\frac{\pi}{2}}\int_{0}^{\pi}\int_{0}^{2\pi}\int_{R}^{\infty}\frac{\delta\left(3+\delta\right)^{\frac{3}{2}}\left(\frac{3+\delta}{2}\right)^{\frac{\delta}{2}} r^{\delta+2}|\sin\varphi\cos^{\delta-1}k\sin^2k|}{r^{q-2}}dr d\theta d\varphi dk \cr
&\leq \|f\|_{L_q^{\infty}} \left\{\frac{2\pi^2\delta\left(3+\delta\right)^{\frac{3}{2}}\left(\frac{3+\delta}{2}\right)^{\frac{\delta}{2}}}{q-\delta-5}\right\} R^{\delta+5-q}\cr
&=\|f\|_{L_q^{\infty}}
\left\{\frac{2^{\frac{2-\delta}{2}}\pi^2(3+\delta)^{\frac{3+\delta}{2}}\delta}{q-\delta-5}\right\}R^{\delta+5-q}.
\end{align*}
Inserting these computations  into (\ref{Lemma22}), we get
\begin{align*}
\rho\left(T_{\delta}+\frac{1}{3+\delta}|U|^2\right)
&\leq \rho R^2+\left\{\frac{2^{\frac{2-\delta}{2}}\pi^2(3+\delta)^{\frac{3+\delta}{2}}\delta}{q-\delta-5}\right\}\|f\|_{L_q^{\infty}}R^{\delta+5-q}.
\end{align*}
Now, take
\[
R^{\delta+3-q}=\left\{\frac{q-\delta-5}{2^{\frac{2-\delta}{2}}\pi^2(3+\delta)^{\frac{3+\delta}{2}}\delta}\right\}\frac{\rho}{\|f\|_{L_q^{\infty}}},
\]
to get
\begin{align*}
\rho\left(T_{\delta}+\frac{1}{3+\delta}|U|^2\right) \leq 2\bigg\{\frac{2^{\frac{2-\delta}{2}}\pi^2(3+\delta)^{\frac{3+\delta}{2}}\delta}{q-\delta-5}\bigg\}^{\frac{2}{q-\delta-3}}\rho^{\frac{\delta+5-q}{\delta+3-q}}\|f\|_{L_q^{\infty}}^{\frac{2}{q-\delta-3}}.
\end{align*}
This implies
\begin{align*}
\rho(T_{\delta}+|U|^2)^{\frac{q-\delta-3}{2}}
&\leq
\{2(3+\delta)\}^{\frac{q-\delta-3}{2}} \left\{\frac{2^{\frac{2-\delta}{2}}\pi^2(3+\delta)^{\frac{3+\delta}{2}}\delta}{q-\delta-5}\right\}
\|f\|_{L_q^{\infty}} \cr
&=
\left\{\frac{2^{\frac{q-2\delta-1}{2}}\pi^2(3+\delta)^{\frac{q}{2}}\delta}{q-\delta-5}\right\}\|f\|_{L_q^{\infty}},
\end{align*}
which completes the proof.
\end{proof}

\begin{lemma}\label{Lemma3}
Assume $\|f\|_{L_q^{\infty}},\, \rho, \,T_{\delta}>0$. Then we have
\begin{align*}
\frac{\rho|U|^{3+\delta+q}}{[(T_{\delta}+|U|^2)T_{\delta}]^{\frac{3+\delta}{2}}}
\leq C_{\delta,q}\|f\|_{L_q^{\infty}}
\end{align*}
where $C_{\delta,q}=2^{\frac{11+2\delta+2q}{2}}\pi^2(3+\delta)^{2+\delta}\delta$.
\end{lemma}
\begin{proof}
For simplicity, we set
\[
A(v,I)=\sqrt{\frac{1}{3+\delta}}|v-U|+\sqrt{\frac{2}{3+\delta}}\,I^{\frac{1}{\delta}}.
\]
We split the macroscopic momentum as
\begin{align*}
\rho|U|
&\leq
\int_{A(v,I)\leq R}f|v|dvdI
+\int_{A(v,I)>R}f|v|dvdI \cr
&\equiv I_1+I_2.
\end{align*}
By H\"{o}lder's inequality,
\begin{align*}
I_1&\leq
\left\{\int_{A(v,I)\leq R}fdvdI\right\}^{1-\frac{1}{q}}
\left\{\int_{A(v,I)\leq R}\left(f^{\frac{1}{q}}|v|\right)^{q}dvdI\right\}^{\frac{1}{q}}
\cr
&\leq
\left\{\int_{\mathbb{R}^3\times \mathbb{R}^+}fdvdI\right\}^{1-\frac{1}{q}}
\left\{\int_{A(v,I)\leq R}f|v|^qdvdI\right\}^{\frac{1}{q}} \cr
&\leq
\rho^{1-\frac{1}{q}}\|f\|_{L_q^{\infty}}^{\frac{1}{q}}	\left\{\int_{A(v,I)\leq R}dvdI\right\}^{\frac{1}{q}}.
\end{align*}
Then, computing similarly as in the previous lemma, we have
\begin{align*}
\int_{A(v,I)\leq R}dvdI&\leq \int_{\frac{1}{3+\delta}|v-U|^2+\frac{2}{3+\delta}I^{2/\delta}\leq R^2}dvdI\leq
\left\{2^{\frac{2-\delta}{2}}\pi^2(3+\delta)^{\frac{1+\delta}{2}}\delta \right\}R^{3+\delta}.
\end{align*}
Therefore, we bound $I_1$ by
\begin{align*}
\rho^{1-\frac{1}{q}}\|f\|_{L_q^{\infty}}^{\frac{1}{q}}\left\{2^{\frac{2-\delta}{2}}\pi^2(3+\delta)^{\frac{1+\delta}{2}}\delta\right\}^{\frac{1}{q}}R^{\frac{3+\delta}{q}}.
\end{align*}
On the other hand, we observe
\begin{align*}
I_2\leq \frac{1}{R}\int_{A(v,I)>R}f|v|\left\{\sqrt{\frac{1}{3+\delta}}|v-U|+\sqrt{\frac{2}{3+\delta}}I^{\frac{1}{\delta}}\right\}dvdI.
\end{align*}
Applying H\"{o}lder's inequality again,
\begin{align*}
I_2
&\leq \frac{\sqrt{2(3+\delta)}}{R}\left\{\int_{\mathbb{R}^3\times\mathbb{R}^+}f\left(\frac{1}{3+\delta}|v|^2+\frac{2}{3+\delta}I^{\frac{2}{\delta}}\right)dvdI\right\}^{\frac{1}{2}} \cr
&\times
\left\{\int_{\mathbb{R}^3\times\mathbb{R}^+}f\left(\frac{1}{3+\delta}|v-U|^2+\frac{2}{3+\delta}I^{\frac{2}{\delta}}\right)dvdI\right\}^{\frac{1}{2}} \cr
&=
\frac{\sqrt{2(3+\delta)}}{R}\left\{\frac{1}{3+\delta}\rho|U|^2+\rho T_{\delta}\right\}^{\frac{1}{2}}\{\rho T_{\delta}\}^{\frac{1}{2}}.
\end{align*}
In conclusion,
\begin{align}\label{keep comming}
\rho |U|
\leq
\rho^{1-\frac{1}{q}}\|f\|_{L_q^{\infty}}^{\frac{1}{q}}\left\{2^{\frac{2-\delta}{2}}\pi^2(3+\delta)^{\frac{1+\delta}{2}}\delta\right\}^{\frac{1}{q}}R^{\frac{3+\delta}{q}}
+\frac{\sqrt{2(3+\delta)}}{R}\rho[(|U|^2+T_{\delta})T_{\delta}]^{\frac{1}{2}}.
\end{align}
The optimizing choice for $R$ then is
\begin{align*}
R^{3+\delta+q}=\frac{[2(3+\delta)]^{\frac{q}{2}}\rho[(|U|^2+T_{\delta})T_{\delta}]^{\frac{q}{2}}}{\left\{2^{\frac{2-\delta}{2}}\pi^2(3+\delta)^{\frac{1+\delta}{2}}\delta\right\}\|f\|_{L_q^{\infty}}},
\end{align*}
for which the right hand side of (\ref{keep comming}) becomes
\begin{align*}
2\left\{2^{\frac{2-\delta}{2}}\pi^2(3+\delta)^{\frac{1+\delta}{2}}\delta\right\}^{\frac{1}{3+\delta+q}}
\{2(3+\delta)\}^{\frac{3+\delta}{2(3+\delta+q)}}\rho^{\frac{2+\delta+q}{3+\delta+q}}[(|U|^2+T_{\delta})T_{\delta}]^{\frac{3+\delta}{2(3+\delta+q)}}\|f\|_{L_q^{\infty}}^{\frac{1}{3+\delta+q}}.
\end{align*}
This gives
\begin{align*}
\frac{\rho|U|^{3+\delta+q}}{[(|U|^2+T_{\delta})T_{\delta}]^{\frac{3+\delta}{2}}}
&\leq 2^{3+\delta+q}\{2(3+\delta)\}^{\frac{3+\delta}{2}}
\left\{2^{\frac{2-\delta}{2}}\pi^2(3+\delta)^{\frac{1+\delta}{2}}\delta\right\}\|f\|_{L_q^{\infty}} \cr
&=
2^{\frac{11+2\delta+2q}{2}}\pi^2(3+\delta)^{2+\delta}\delta\|f\|_{L_q^{\infty}}.
\end{align*}
\end{proof}
%
%
%
%
%
\section{Solution space and approximate scheme}
We set up our solution space $\Omega$:
\begin{align*}
\begin{split}
\Omega&=\big\{f\in C_+\big([0,T]; \|\cdot\|_{L^{\infty}_q}\big)~\big|~
f \mbox{ satisfies }  (\mathcal{A}1) \mbox{ and } (\mathcal{A}2)~ \big\},
\end{split}
\end{align*}
where properties ($\mathcal{A}1$) and ($\mathcal{A}2$) are
\begin{itemize}
\item ($\mathcal{A}1$): There exists a constant $C_1>0$  such that
\begin{align*}
\|f(t)\|_{L^{\infty}_q} \leq e^{C_1t}\|f_0\|_{L^{\infty}_q},\quad \mbox{for } t\in[0,T].
\end{align*}
\item ($\mathcal{A}2$): There exist positive constants $C_{T,f_0}$, $C_{T,f_0,\delta}$ and $C_{T,f_0,\delta,q}$ such that
\begin{align*}
	&(1)~\rho(x,t) \geq C_{T,f_0},\cr
	&(2)~T_{\delta}(x,t)\geq C_{T,f_0,\delta},\cr
	&(3)~\rho+|U|+T_{\delta}\leq C_{T,f_0,\delta,q}.
\end{align*}
\end{itemize}
We consider the following iteration scheme: $(n \geq 1)$
\begin{align}\label{Mildform}
\begin{split}
\partial_tf^{n+1}+v\cdot\nabla_xf^{n+1}&=A_{\nu,\theta}\big(\mathcal{M}_{\nu,\theta}(f^{n})-f^{n+1}\big),\cr
f^{n+1}(0)&=f_0.
\end{split}
\end{align}
We set $f^0=0$ and $\mathcal{M}(f^0)=0$, so that
\begin{align*}
\partial_tf^{1}+v\cdot\nabla_xf^{1}&+A_{\nu,\theta}f^{1}=0, \cr
f^{1}(0)&=f_0.
\end{align*}
Our first goal is to show that $\{f^n\}$ lies in  $\Omega$ for all $n\geq 0$.
We start with the following estimates on the polyatomic Gaussian.
\begin{proposition}\label{MconF_theta}
Suppose $f \in \Omega$, there exists a constant $C_{\mathcal{M}}$ depending on $\nu, \delta,\theta$ and $q$ such that
\begin{align*}
\|\mathcal{M}_{\nu,\theta}(f)\|_{L_{q}^{\infty}} \leq C_{\mathcal{M}}\|f\|_{L_{q}^{\infty}}.
\end{align*}
\end{proposition}
\begin{remark}
$C_{\mathcal{M}}$ blows up as $\theta$ tends to 0. See the end of the proof.
\end{remark}
\begin{proof}
We will show that $\mathcal{M}_{\nu,\theta}(f)$, $|v|^{q}\mathcal{M}_{\nu,\theta}(f)$ and $I^{\frac{q}{\delta}}\mathcal{M}_{\nu,\theta}(f) $  are controlled by $\|f\|_{L_q^{\infty}}$. \newline

\noindent(a) {\bf The estimate for $\mathcal{M}_{\nu,\theta}(f)$ :}
We first recall Lemma \ref{Temperature} to observe
\begin{align}\label{less0}
\frac{1}{2}(v-U)^{\top}\mathcal{T}_{\nu,\theta}^{-1}(v-U)+\frac{I^{\frac{2}{\delta}}}{T_{\theta}}\geq
\frac{3}{2C_{\nu}\big\{3+\delta(1-\theta)\big\}}\frac{|v-U|^2}{T_{\delta}}+\frac{I^{\frac{2}{\delta}}}{T_{\theta}}\geq 0
\end{align}
for $f\in\Omega$.
Hence we have
\begin{align}\label{less1}
\exp\left(-\frac{1}{2}(v-U)^{\top}\mathcal{T}_{\nu,\theta}^{-1}(v-U)-\frac{I^{\frac{2}{\delta}}}{T_{\theta}}\right) \leq 1.
\end{align}
Using this and Lemma \ref{Temperature} and Lemma \ref{Lemma1}, we have
\begin{align*}
\mathcal{M}_{\nu,\theta}(f)
&\leq \frac{\rho \Lambda_{\delta}}{\sqrt{\det (2\pi \mathcal{T}_{\nu,\theta})} (T_{\theta})^{\frac{\delta}{2}}}\cr
&\leq
\frac{1}{(2\pi)^{3/2}}\frac{1}{\theta^{\frac{3+\delta}{2}}}\frac{\rho}{T_{\delta}^{\frac{3+\delta}{2}}} \cr
&\leq
\frac{1}{(2\pi)^{3/2}}\frac{1}{\theta^{\frac{3+\delta}{2}}}\left\{2^{\frac{7}{2}}(3+\delta)^{\frac{1+\delta}{2}}\pi^2\delta\right\}\|f\|_{L_{q}^{\infty}}\cr
&\equiv \frac{C_0}{\theta^{\frac{3+\delta}{2}}}\|f\|_{L_q^{\infty}}
.
\end{align*}

\noindent$(b)$ {\bf The estimate for $\mathcal{M}_{\nu,\theta}(f)|v|^{q}$:} We divide it into the estimates of $|U|^{q}\mathcal{M}_{\nu,\theta}(f)$
and $|v-U|^{q}\mathcal{M}_{\nu,\theta}(f)$.\newline

\noindent$(b_1)$ $|U|^{q}\mathcal{M}_{\nu,\theta}(f)$: We use (\ref{less1}) and Lemma \ref{Temperature} to compute
\begin{align*}
|U|^q\mathcal{M}_{\nu,\theta}(f)\leq |U|^q\frac{\rho \Lambda_{\delta}}{\sqrt{\det (2\pi \mathcal{T}_{\nu,\theta})} (T_{\theta})^{\frac{\delta}{2}}}
\leq
\frac{1}{(2\pi)^{3/2}}\frac{1}{\theta^{\frac{3+\delta}{2}}}|U|^q \frac{\rho}{T_{\delta}^{\frac{3+\delta}{2}}}.
\end{align*}
We divide this estimate into two cases. In the case of $|U|<T_{\delta}^{\frac{1}{2}}$, we have from Lemma \ref{Lemma2} that
\begin{align*}
|U|^q\frac{\rho}{T_{\delta}^{\frac{3+\delta}{2}}}
\leq
\rho (T_{\delta}+|U|^2)^{\frac{q-3-\delta}{2}}
\leq \bigg\{\frac{2^{\frac{q-2\delta-1}{2}}\pi^2(3+\delta)^{\frac{q}{2}}\delta}{q-\delta-5}\bigg\}\|f\|_{L_q^{\infty}}.
\end{align*}
On the other hand, in the case of $|U|\geq T_{\delta}^{\frac{1}{2}}$, we have from Lemma \ref{Lemma3} that
\begin{align*}
|U|^q\frac{\rho}{T_{\delta}^{\frac{3+\delta}{2}}}
\leq \frac{\rho|U|^{q+3+\delta}}{|U|^{3+\delta}T_{\delta}^{\frac{3+\delta}{2}}}
\leq
2^{\frac{3+\delta}{2}}\frac{\rho|U|^{q+3+\delta}}{[(T_{\delta}+|U|^2)T_{\delta}]^{\frac{3+\delta}{2}}} \leq 2^{\frac{14+3\delta+2q}{2}}\pi^2(3+\delta)^{2+\delta}\delta\|f\|_{L_q^{\infty}}.
\end{align*}
These two estimates give
\begin{align*}
|U|^q \mathcal{M}_{\nu,\theta}(f)\leq
\frac{C_1}{\theta^{\frac{3+\delta}{2}}}\|f\|_{L_q^{\infty}}.
\end{align*}
for
\[
C_1=\bigg\{\frac{2^{\frac{q-2\delta-4}{2}}\sqrt{\pi}(3+\delta)^{\frac{q}{2}}\delta}{q-\delta-5}\bigg\}
+2^{\frac{11+3\delta+2q}{2}}\sqrt{\pi}(3+\delta)^{2+\delta}\delta.
\]
$(b_2)$  {\bf $|v-U|^q\mathcal{M}_{\nu,\theta}(f)$:}
From (\ref{less0}) and Lemma \ref{Temperature}, we have
\begin{align*}
&|v-U|^{q} \mathcal{M}_{\nu,\theta}(f)\cr
&\quad\leq
\frac{1}{(2\pi)^{3/2}}\frac{1}{\theta^{\frac{3+\delta}{2}}}|v-U|^q\frac{\rho}{T_{\delta}^{\frac{3+\delta}{2}}}
\exp\Bigg(-\frac{3}{2C_{\nu}\big\{3+\delta(1-\theta)\big\}}\frac{|v-U|^2}{T_{\delta}}\Bigg) \cr
&\quad=\frac{1}{(2\pi)^{3/2}}
\frac{1}{\theta^{\frac{3+\delta}{2}}}T_{\delta}^{\frac{q}{2}}\frac{\rho}{T_{\delta}^{\frac{3+\delta}{2}}}\left(\frac{|v-U|^2}{T_{\delta}} \right)^{\frac{q}{2}}
\exp\Bigg(-\frac{3}{2C_{\nu}\big\{3+\delta(1-\theta)\big\}}\frac{|v-U|^2}{T_{\delta}}\Bigg) \cr
&\quad\equiv \frac{C_2}{\theta^{\frac{3+\delta}{2}}}\rho T_{\delta}^{\frac{q-3-\delta}{2}},
\end{align*}
where
\[
C_2=\frac{1}{(2\pi)^{3/2}} \sup_{x\geq0}\big(x^{q/2}e^{-x}\big)\left\{\frac{2C_{\nu}(3+\delta(1-\theta))}{3}\right\}^{q/2}.
\]
This, combined with Lemma \ref{Lemma2} implies
\begin{align*}
|v-U|^{q} \mathcal{M}_{\nu,\theta}(f) &\leq \frac{C_2}{\theta^{\frac{3+\delta}{2}}}\rho (T_{\delta}+|U|^2)^{\frac{q-3-\delta}{2}} \cr
&\leq
\frac{C_2}{\theta^{\frac{3+\delta}{2}}}
\left\{\frac{2^{\frac{q-2\delta-1}{2}}\pi^2(3+\delta)^{\frac{q}{2}}\delta}{q-\delta-5}\right\}\|f\|_{L_q^{\infty}}\cr
&\equiv
\frac{C_3}{\theta^{\frac{3+\delta}{2}}}
\|f\|_{L_q^{\infty}}.
\end{align*}
(c) {\bf The estimate for $I^{\frac{q}{\delta}}\mathcal{M}_{\nu,\theta}(f) $:} Again from (\ref{less0}), we have
\begin{align*}
\frac{1}{2}(v-U)^{\top}\mathcal{T}_{\nu,\theta}^{-1}(v-U)+\frac{I^{\frac{2}{\delta}}}{T_{\theta}}\geq
\frac{\delta}{\delta+3(1-\theta)}\frac{I^{\frac{2}{\delta}}}{T_{\delta}},
\end{align*}
so that $I^{\frac{q}{\delta}}\mathcal{M}_{\nu,\theta}(f)$ is estimated as follows:
\begin{align*}
I^{\frac{q}{\delta}}\mathcal{M}_{\nu,\theta}(f)
&\leq
\frac{1}{\sqrt{(2\pi)^3}}I^{\frac{q}{\delta}}\frac{1}{\theta^{\frac{3+\delta}{2}}}\frac{\rho}{T_{\delta}^{\frac{3+\delta}{2}}}\exp\bigg(-\frac{\delta}{\delta+3(1-\theta)}\frac{I^{\frac{2}{\delta}}}{T_{\delta}}\bigg)
\cr
&=
\frac{1}{\sqrt{(2\pi)^3}}\frac{1}{\theta^{\frac{3+\delta}{2}}}T_{\delta}^{\frac{q}{2}}\frac{\rho}{T_{\delta}^{\frac{3+\delta}{2}}}\left(\frac{I^{\frac{2}{\delta}}}{T_{\delta}} \right)^{\frac{q}{2}}
\exp\bigg(-\frac{\delta}{\delta+3(1-\theta)}\frac{I^{\frac{2}{\delta}}}{T_{\delta}}\bigg)
\cr
&\equiv \frac{C_4}{\theta^{\frac{3+\delta}{2}}}\rho T_{\delta}^{\frac{q-3-\delta}{2}},
\end{align*}
where
\[
C_4=\frac{1}{(2\pi)^{3/2}}\sup_{x\geq0}(x^{q/2}e^{-x})\left(\frac{\delta+3(1-\theta)}{\delta}\right)^{q/2}.
\]
Then, in view of Lemma \ref{Lemma2}, we derive
\begin{align*}
I^{\frac{q}{\delta}}\mathcal{M}_{\nu,\theta}(f)
&\leq \frac{C_4}{\theta^{\frac{3+\delta}{2}}}\rho (T_{\delta}+|U|^2)^{\frac{q-3-\delta}{2}} \cr
&\leq
\frac{C_4}{\theta^{\frac{3+\delta}{2}}}\left\{\frac{2^{\frac{q-2\delta-1}{2}}\pi^2(3+\delta)^{\frac{q}{2}}\delta}{q-\delta-5}\right\}\|f\|_{L_q^{\infty}}\cr
&\equiv
\frac{C_5}{\theta^{\frac{3+\delta}{2}}}\|f\|_{L_q^{\infty}}.
\end{align*}
Finally, we combine  $(a)$, $(b)$ and $(c)$ to conclude that
\begin{align*}
\|\mathcal{M}_{\nu,\theta}(f)\|_{L_{q}^{\infty}} \leq \frac{C_{\nu,\delta,\theta,q}}{\theta^{\frac{3+\delta}{2}}} \|f\|_{L_{q}^{\infty}},
\end{align*}
where
\[
C_{\nu,\delta,\theta,q}=C_0+C_1+C_3+C_5.
\]
Note that $\displaystyle\max_{0\leq \theta\leq 1}C_{\nu,\delta,\theta,q}<\infty$.
\end{proof}
\begin{proposition}\label{prop2.2}
$f^n$ lies in $\Omega$ for all $n>0$.
That is, $f^n$ satisfies
\begin{itemize}
\item $(\mathcal{A}1)$: $f^n$ is uniformly bounded in $\|\cdot\|_{L^{\infty}_q}$
\begin{align*}
\|f^n\|_{L^{\infty}_q} \leq e^{C_{1}t}\|f_0\|_{L^{\infty}_q},
\end{align*}
where $C_1=A_{\nu,\theta}\left(C_{\mathcal{M}}-1\right)$.
\item $(\mathcal{A}2)$:  There exist positive constants $C_{T,f_0}$, $C_{T,f_0,\delta}$ and $C_{T,f_0,\delta,q}$ such that
\begin{align*}
	&(1)~\rho^n(x,t) \geq C_{T,f_0},\cr
	&(2)~T_{\delta}^n(x,t)\geq C_{T,f_0,\delta},\cr
	&(3)~\rho^n+|U^n|+T_{\delta}^n\leq C_{T,f_0,\delta,q}.
\end{align*}
\end{itemize}
\end{proposition}
\begin{proof}
We proceed by induction. The properties are trivially satisfied for $n=0$. Assume that $f^n\in\Omega$.  We prove that $f^{n+1}$ also satisfies$(\mathcal{A}1)$ and $(\mathcal{A}2)$.\newline

\noindent $(\mathcal{A}1)$ We write (\ref{Mildform}) in the mild form:
\begin{align*}
f^{n+1}(t,x,v,I)=e^{-A_{\nu,\theta}t}f_0(x-vt,v,I)+A_{\nu,\theta}\int_0^te^{-A_{\nu,\theta}(t-s)}\mathcal{M}_{\nu,\theta}(f^n)(x-(t-s)v,v,s,I)ds
\end{align*}
and take $\|\cdot\|_{L^{\infty}_q}$ on both sides,
\begin{align}\label{mildform}
\|f^{n+1}(t)\|_{L^{\infty}_q}
\leq
e^{-A_{\nu,\theta}t}\|f_0\|_{L^{\infty}_q}+A_{\nu,\theta}\int_0^t e^{-A_{\nu,\theta}(t-s)}\|\mathcal{M}_{\nu,\theta}(f^n)(s)\|_{L^{\infty}_q}ds.
\end{align}
Since $f^n\in \Omega$, we can apply  Proposition \ref{MconF_theta} to estimate
\begin{align*}
\begin{split}
A_{\nu,\theta}\int_0^t e^{-A_{\nu,\theta}(t-s)}\|\mathcal{M}_{\nu,\theta}(f^n)(s)\|_{L^{\infty}_q}ds
&\leq A_{\nu,\theta}\int_0^t e^{-A_{\nu,\theta}(t-s)}C_{\mathcal{M}}\|f^n(s)\|_{L^{\infty}_q}ds \cr
&\leq A_{\nu,\theta}\int_0^t e^{-A_{\nu,\theta}(t-s)}C_{\mathcal{M}}e^{C_1s}\|f_0\|_{L^{\infty}_q}ds \cr
&= \frac{A_{\nu,\theta}C_{\mathcal{M}}}{C_1+A_{\nu,\theta}}(e^{C_1t}-e^{-A_{\nu,\theta}t})\|f_0\|_{L^{\infty}_q},
\end{split}
\end{align*}
where we used $\|f^n\|_{L^{\infty}_q} \leq e^{C_{1}t}\|f_0\|_{L^{\infty}_q}$.
Plugging this estimate into (\ref{mildform}), we get
\begin{align*}
\|f^{n+1}(t)\|_{L^{\infty}_q} \leq e^{C_1t}\|f_0\|_{L^{\infty}_q},
\end{align*}
since $(A_{\nu,\theta}C_{\mathcal{M}})/(C_1+A_{\nu,\theta})=1$.\newline

\noindent $(\mathcal{A}2)$ By the nonnegativity of polyatomic Gaussian $\mathcal{M}_{\nu,\theta}(f^n)$, we have
from the above mild form
\[
f^{n+1} \geq e^{-A_{\nu,\theta}t}f_0(x-vt,v,I).
\]
Integrating in $v$ and $I$ on both sides, and recalling the lower bound assumption imposed on $f_0$,
\begin{align*}
\rho^{n+1}=\int_{\mathbb{R}^3 \times \mathbb{R}^+}f^{n+1} dvdI \geq e^{-A_{\nu,\theta}t}\int_{\mathbb{R}^3 \times \mathbb{R}^+}
f_0(x-vt,v,I) dvdI \geq C_{f_0}e^{-A_{\nu,\theta}t}.
\end{align*}
Hence, combining the above results and Lemma \ref{Lemma1} gives
\begin{align*}
C_{f_0}e^{-A_{\nu,\theta}t}\leq\rho^{n+1} \leq
C_{\delta}\|f^{n+1}\|_{L_q^{\infty}}\big\{T_{\delta}^{n+1}\big\}^{\frac{3+\delta}{2}} \leq C_{\delta}e^{C_1t}\|f_0\|_{L_q^{\infty}}\big\{T_{\delta}^{n+1}\big\}^{\frac{3+\delta}{2}}.
\end{align*}
Therefore,
\begin{align*}
	T^{n+1}_{\delta} \geq \Bigg(\frac{C_{f_0}e^{-A_{\nu,\theta}t}}{C_{\delta}e^{C_1t}\|f_0\|_{L_q^{\infty}}}\Bigg)^{\frac{2}{3+\delta}}\geq C_{T,f_0,\delta}.
\end{align*}
The estimate $(\mathcal{A}2)$ $(3)$ follows immediately from the above lower bound for $\rho^{n+1}$ and Lemma \ref{Lemma2}.
This completes the proof.
\end{proof}

\section{Lipschitz continuity of $\mathcal{M}_{\nu,\theta}$}
\begin{proposition}\label{Lip_int}
Let $f$ and $g$ lie in $\Omega$. Then $\mathcal{M}_{\nu,\theta}$ satisfies the following continuity property:
\begin{align*}
\|\mathcal{M}_{\nu,\theta}(f)-\mathcal{M}_{\nu,\theta}(g)\|_{L_q^{\infty}} \leq C_{Lip}\|f-g\|_{L_q^{\infty}}
\end{align*}
for some constant $C_{Lip}$ depending on $T, \delta, \theta, q$ and $f_0$.
\end{proposition}
\begin{proof}
For the proof of this proposition, we set the transitonal macroscopic fields $\rho_{\eta}, U_{\eta}, \mathcal{T}_{\nu,\theta\eta}$, $T_{I,\delta \eta}$:
\begin{align*}
(\rho_{\eta}, U_{\eta}, \mathcal{T}_{\nu,\theta\eta}, T_{I,\delta {\eta}}) ={\eta}(\rho_f, U_f, \mathcal{T}_{\nu,\theta f}, T_{I,\delta f})+(1-{\eta})(\rho_g, U_g, \mathcal{T}_{\nu,\theta g}, T_{I,\delta g})
\end{align*}
and define the transitional polyatomic Gaussian:
\begin{align*}
\mathcal{M}_{\nu,\theta}({\eta})&=\frac{\rho_{\eta}\Lambda_{\delta}}{\sqrt{\det(2\pi\mathcal{T}_{\nu,\theta \eta})} (T_{\theta {\eta}})^\frac{\delta}{2}}\exp\bigg(-\frac{1}{2}(v-U_{\eta})^{\top}\mathcal{T}_{\nu,\theta\eta}^{-1}(v-U_{\eta})-\frac{I^{\frac{2}{\delta}}}{T_{\theta {\eta}}}\bigg).
\end{align*}

Applying Taylor's theorem, we expand
\begin{align}\label{Msplit}
\begin{split}
\mathcal{M}_{\nu,\theta}(f)-\mathcal{M}_{\nu,\theta}(g)
&=(\rho_f-\rho_g)\int_{0}^{1}\frac{\partial\mathcal{M}_{\nu,\theta}({\eta}) }{\partial\rho_{\eta}}d{\eta} \cr
&+ (U_f-U_g)\int_{0}^{1}\frac{\partial\mathcal{M}_{\nu,\theta}({\eta}) }{\partial U_{\eta}}d{\eta} \cr
&+(\mathcal{T}_{\nu,\theta f}-\mathcal{T}_{\nu,\theta g})\int_{0}^{1}\frac{\partial\mathcal{M}_{\nu,\theta}({\eta}) }{\partial\mathcal{T}_{\nu,\theta\eta}}d{\eta} \cr
&+(T_{I,\delta f}-T_{I,\delta g})\int_{0}^{1}\frac{\partial\mathcal{M}_{\nu,\theta}({\eta}) }{\partial T_{I,\delta {\eta}}}d{\eta} \cr
&=I_1+I_2+I_3+I_4.
\end{split}
\end{align}
We only consider $I_3$. Other terms can be treated in a similar and simpler manner.
Recalling the definition of $\mathcal{T}_{\nu,\theta}$, we see that
\begin{align*}
&\rho_f\mathcal{T}_{\nu,\theta f}-\rho_g\mathcal{T}_{\nu,\theta g}\cr
&=(1-\theta)\rho_f\{(1-\nu)T_{trf}Id+\nu\Theta_f\}+\theta\rho_fT_{\delta f}Id-(1-\theta)\rho_g\{(1-\nu)T_{trg}Id+\nu\Theta_g\}-\theta \rho_gT_{\delta g}Id\cr
&=(1-\theta)
\int_{\mathbb{R}^3\times\mathbb{R}^+}f\Big\{\frac{(1-\nu)}{3}|v-U_f|^2Id+\nu(v-U_f)\otimes(v-U_f)\Big\}dvdI \cr
&+\frac{\theta}{3+\delta}
\int_{\mathbb{R}^3\times\mathbb{R}^+}f\big\{|v-U_f|^2+2I^{\frac{2}{\delta}}\big\}IddvdI	\cr
&-(1-\theta)
\int_{\mathbb{R}^3\times\mathbb{R}^+}g\Big\{\frac{(1-\nu)}{3}|v-U_g|^2Id+\nu(v-U_g)\otimes(v-U_g)\Big\}dvdI \cr
&-\frac{\theta}{3+\delta}
\int_{\mathbb{R}^3\times\mathbb{R}^+}g\Big\{|v-U_g|^2+2I^{\frac{2}{\delta}}\Big\}IddvdI,
\end{align*}
which can be rearranged as
\begin{align*}
&\Big\{(1-\theta)\frac{1-\nu}{3}+\frac{\theta}{3+\delta}\Big\}\bigg(\int_{\mathbb{R}^3\times\mathbb{R}^+}(f|v-U_f|^2-g|v-U_g|^2) dvdI\bigg)Id \cr
&\qquad+(1-\theta)\nu
\int_{\mathbb{R}^3\times\mathbb{R}^+}\{f(v-U_f)\otimes(v-U_f)-g(v-U_g)\otimes(v-U_g)\}dvdI\cr
&\qquad+\frac{2\theta}{3+\delta}\bigg(\int_{\mathbb{R}^3\times\mathbb{R}^+}(f-g)I^{\frac{2}{\delta}}dvdI\bigg)Id \cr
&\qquad\equiv T_1+T_2+T_3.
\end{align*}
For $T_1$, we note that
\begin{align*}
\int_{\mathbb{R}^3 \times \mathbb{R}^+}f|v-U_f|^2-g|v-U_g|^2 dvdI =
\int_{\mathbb{R}^3 \times \mathbb{R}^+} (f-g)|v|^2dvdI-\big(\rho_f|U_f|^2-\rho_g|U_g|^2\big).
\end{align*}
The first term is clearly bounded by $C\|f-g\|_{L^{\infty}_q}~ (q>3)$.
For the second term, we compute
\begin{align*}
\Big|\rho_f|U_f|^2-\rho_g|U_g|^2\Big|&=\left|\frac{\rho^2_f|U_f|^2-\rho^2_g|U_g|^2}{\rho_f}+\frac{\rho^2_g}{\rho_f}|U_g|^2
-\rho_g|U_g|^2\right|\cr
&\leq\frac{1}{\rho_f}\Big(|\rho_f U_f|+|\rho_gU_g|\Big)\big|\rho_fU_f-\rho_gU_g\big|
+\frac{1}{\rho_f\rho_g}|\rho_gU_g|^2|\rho_f-\rho_g |.
\end{align*}
Now, since $f, g\in \Omega$, we have from $(\mathcal{A}1)$
\[
|\rho_fU_f|+|\rho_gU_g|\leq C_{\delta,q}\big(\|f\|_{L^{\infty}_q}+\|g\|_{L^{\infty}_q}\big)\leq C_{\delta,q}e^{C_{1}t}\|f_0\|_{L^{\infty}_q}
\]
and $\rho_f, \rho_g\geq C_{T,f_0}$. Therefore,
\begin{align*}
\Big|\rho_f|U_f|^2-\rho_g|U_g|^2\Big|&\leq
C_{T,f_0,\delta,q}\left\{\big|\rho_fU_f-\rho_gU_g\big|
+|\rho_f-\rho_g |\right\}\cr
&\leq C_{T,f_0,\delta,q}\int_{\mathbb{R}^3\times\mathbb{R}^+} |f-g|(1+|v|^2)dvdI\cr
&\leq C_{T,f_0,\delta,q}\|f-g\|_{L^{\infty}_q}.
\end{align*}
The estimate for $T_2$ can be derived from a similar computation using the following identity:
\begin{align*}
\{(1-\theta)\nu\}^{-1}T_2
&=
\int_{\mathbb{R}^3\times\mathbb{R}^+}(f-g)v\otimes vdvdI \cr
&-
\frac{1}{\rho_f}\rho_fU_f\otimes(\rho_fU_f-\rho_gU_g)-\frac{\rho_g}{\rho_f}(\rho_fU_f-\rho_gU_g)\otimes U_g\cr
&+\frac{1}{\rho_f\rho_g}(\rho_f-\rho_g)\big\{\rho_gU_g\big\}\otimes \big\{\rho_gU_g\big\}.
\end{align*}
We omit the estimate for $T_3$. What we have shown so far is
\[
|\rho_f\mathcal{T}_{\nu,\theta f}-\rho_g\mathcal{T}_{\nu,\theta g}|\leq C_{T,f_0,\delta,\theta,q}\|f-g\|_{L_q^{\infty}}.
\]
Therefore,
\begin{align*}
|\mathcal{T}_{\nu,\theta f}-\mathcal{T}_{\nu,\theta g}| &\leq \frac{1}{\rho_f}|\rho_f\mathcal{T}_{\nu,\theta f}-\rho_g\mathcal{T}_{\nu,\theta g}| +\frac{1}{\rho_f}|\rho_f-\rho_g||\mathcal{T}_{\nu,\theta g}| \leq
C_{T,f_0,\delta,\theta,q}\|f-g\|_{L_q^{\infty}},
\end{align*}
where we used Lemma \ref{Temperature} and Property $(\mathcal{A}2) (1)$ of $\Omega$ as
\begin{align*}
\rho_f\geq C_{T,f_0},
\end{align*}
and
\begin{align}\label{rec}
\begin{split}
|\mathcal{T}^{ij}_{\nu,\theta g}|
&\leq
\frac{1-\theta}{\rho_{g}}\left|\int_{\mathbb{R}^3\times\mathbb{R}^+}g\left\{\frac{(1-\nu)}{3}|v-U_{g}|^2\delta_{ij}+\nu(v-U_{g})_i(v-U_{g})_j\right\}dvdI\right| \cr
&+
\frac{\theta}{\rho_{g}}\left|\int_{\mathbb{R}^3\times\mathbb{R}^+}g\left\{\frac{1}{3+\delta}\left(|v-U_{g}|^2+2I^{\frac{2}{\delta}}\right)\delta_{ij}\right\}dvdI\right| \cr
&\leq(3+\delta-2\theta-\delta\theta) T_{\delta g}.
\end{split}
\end{align}
We now move on to the estimate of the integral in $I_3$. A straightforward computation gives
\begin{align*}
&\int_{0}^{1}\frac{\partial\mathcal{M}_{\nu,\theta}({\eta}) }{\partial \mathcal{T}_{\nu,\theta\eta}}d{\eta} \cr
&\quad=
\int_{0}^{1}\frac{1}{2}
\left[-\frac{1}{\det\mathcal{T}_{\nu,\theta\eta}}\frac{\partial\det\mathcal{T}_{\nu,\theta\eta}}{\partial\mathcal{T}_{{\nu,\theta\eta}ij}}+(v-U_{\eta})^{\top}\mathcal{T}_{\nu,\theta\eta}^{-1}\left(\frac{\partial\mathcal{T}_{\nu,\theta\eta}}{\partial\mathcal{T}_{{\nu,\theta\eta}ij}}\right)\mathcal{T}_{\nu,\theta\eta}^{-1}(v-U_{\eta})\right]\mathcal{M}_{\nu,\theta}({\eta})d{\eta}.
\end{align*}
Before proceeding further, we establish the following claims: For $f,g\in \Omega$, we have
\begin{align*}
(\mathcal{F}_1) &:~ \left|(v-U_{\eta})^{\top}\mathcal{T}_{\nu,\theta\eta}^{-1}\left(\frac{\partial\mathcal{T}^{ij}_{\nu,\theta\eta}}{\partial\mathcal{T}_{{\nu,\theta\eta}}}\right)\mathcal{T}_{\nu,\theta \eta}^{-1}(v-U_{\eta})\right|\leq\frac{|v-U_{\eta}|^2}{[\theta\{\eta T_{\delta f}+(1-\eta)T_{\delta g}\}]^2}  \cr
(\mathcal{F}_2) &:~\left|\det\mathcal{T}_{\nu,\theta\eta}\right| \geq
 \theta^3\big\{\eta T_{\delta f}+(1-\eta)T_{\delta g} \big\}^3\cr
(\mathcal{F}_3) &:~\left|\frac{\partial\det\mathcal{T}_{\nu,\theta\eta}}{\partial\mathcal{T}^{ij}_{\nu,\theta\eta }}\right|
\leq  2(3+\delta-2\theta-\delta\theta)^2\{\eta T_{\delta f}+(1-\eta)T_{\delta g}\}^2.\cr
\end{align*}

\noindent$\bullet~(\mathcal{F}_1)$: Let $D_{ij}$ denote a $n \times n$ matrix whose $ij$th and $ji$th entries are 1 and the remaining entries are 0. Then,
\[
\left|X^{\top}\left(\frac{\partial\mathcal{T}_{\nu,\theta\eta}}{\partial\mathcal{T}^{ij}_{\nu,\theta\eta }}\right)Y\right|=|X^{\top}D_{ij}Y| \leq |X_iY_j+X_jY_i|\leq|X||Y|.
\]
Thus we have
\begin{align*}
\left|(v-U_{\eta})^{\top}\mathcal{T}_{\nu,\theta\eta}^{-1}\left(\frac{\partial\mathcal{T}_{\nu,\theta\eta}}{\partial\mathcal{T}^{ij}_{{\nu,\theta\eta}}}\right)\mathcal{T}_{\nu,\theta \eta}^{-1}(v-U_{\eta})\right|
&\leq
|(v-U_{\eta})^{\top}\mathcal{T}_{\nu,\theta \eta}^{-1}||\mathcal{T}_{\nu,\theta \eta}^{-1}(v-U_{\eta})|.
\end{align*}
Now  we use Lemma \ref{Temperature}:
\begin{align}\label{31}
\mathcal{T}_{\nu,\theta\eta}=\eta\mathcal{T}_{\nu,\theta f}+(1-\eta)\mathcal{T}_{\nu,\theta g} \geq  \left[\theta\{\eta T_{\delta f}+(1-\eta)T_{\delta g}\}\right]Id
\end{align}
to compute
\begin{align*}
|(v-U_{\eta})^{\top}\mathcal{T}_{\nu,\theta \eta}^{-1}|&\leq \sup_{|Y|\leq 1}|(v-U_{\eta})^{\top}\mathcal{T}_{\nu,\theta \eta}^{-1}Y|\cr
&\leq \sup_{|Y|\leq 1}\frac{|v-U_{\eta}||Y|}{\theta \{\eta T_{\delta f}+(1-\eta)T_{\delta g}\}}\cr
&\leq \frac{|v-U_{\eta}|}{\theta\{\eta T_{\delta f}+(1-\eta)T_{\delta g}\}}.
\end{align*}
Likewise,
\begin{align*}
|\mathcal{T}_{\nu,\theta \eta}^{-1}(v-U_{\eta})|\leq \frac{|v-U_{\eta}|}{\theta\{\eta T_{\delta f}+(1-\eta)T_{\delta g}\}},
\end{align*}
which gives the desired estimate.\newline
$\bullet~(\mathcal{F}_2)$:
By (\ref{31}), we have
\[
\det\mathcal{T}_{\nu,\theta\eta}\geq \theta^3\big\{\eta T_{\delta f}+(1-\eta)T_{\delta g} \big\}^3.
\]
$\bullet~(\mathcal{F}_3)$:
We only prove the case: $(i,j)=(1,2)$. An explicit calculation gives
\begin{align}\label{F3}
\frac{\partial\det\mathcal{T}_{\nu,\theta\eta}}{\partial\mathcal{T}^{12}_{\nu,\theta\eta}}
=\mathcal{T}^{23}_{\nu,\theta\eta}\mathcal{T}^{31}_{\nu,\theta\eta}-\mathcal{T}^{33}_{\nu,\theta\eta}
\mathcal{T}^{21}_{\nu,\theta\eta}.
\end{align}
Then we recall (\ref{rec}) to derive
\begin{align*}
|\mathcal{T}^{ij}_{\nu,\theta\eta }|=|\eta\mathcal{T}^{ij}_{\nu,\theta f}+(1-\eta)\mathcal{T}^{ij}_{\nu,\theta g}|\leq (3+\delta-2\theta-\delta\theta)|\eta T_{\delta f}+(1-\eta)T_{\delta g}|.
\end{align*}
Therefore, (\ref{F3}) leads to
\[
\bigg|\frac{\partial\det\mathcal{T}_{\nu,\theta\eta}}{\partial\mathcal{T}^{12}_{\nu,\theta\eta}}\bigg|\leq 2(3+\delta-2\theta-\delta\theta)^2\{\eta T_{\delta f}+(1-\eta)T_{\delta g}\}^2.
\]
This ends the proof of the claims.\newline

\noindent Using these claims, we estimate the integral in $I_3$ as ($T_{\delta\eta}=\eta T_{\delta f}+(1-\eta)T_{\delta g}$.)
\begin{align}\label{hard}
\int_{0}^{1}\frac{\partial\mathcal{M}_{\nu,\theta}({\eta}) }{\partial \mathcal{T}_{\nu,\theta\eta}}d{\eta}
\leq C_{\delta,\theta}\int_{0}^{1} \bigg[\frac{1}{T_{\delta\eta}}+\frac{|v-U_{\eta}|^2}{T_{\delta\eta}^2} \bigg]\mathcal{M}_{\nu,\theta}({\eta})d{\eta}.
\end{align}
Now, since $f,g\in \Omega$, we have
\begin{align*}
 C_{T,f_0,\delta,1}\leq T_{\delta\eta}=\eta T_{\delta f}+(1-\eta)T_{\delta g}\leq C_{T,f_0,\delta,2},
\end{align*}
and
\begin{align*}
\mathcal{M}_{\nu,\theta}(\eta)&\leq
C_{T,f_0,\delta,q}\frac{\rho_{\eta}}{\theta^{\frac{3+\delta}{2}}T_{\delta\eta}^{\frac{3+\delta}{2}}}\cr
&\times \exp\left(-\frac{3}{2C_{\nu}\{3+\delta(1-\theta)\}}\frac{|v-U_{\eta}|^2}{T_{\delta\eta}}\right)
\exp\left(-\frac{\delta}{\delta+3(1-\theta)}\frac{I^{\frac{2}{\delta}}}{T_{\delta\eta}}\right)
\cr
&\leq C_{T,f_0,\delta,\theta,q}e^{-C_{T,f_0,\delta,\theta,q}\big(|v-U_{\eta}|^2+I^{2/\delta}\big)}.
\end{align*}
Therefore, we can proceed further from (\ref{hard}) as
\begin{align*}
\int_{0}^{1}\frac{\partial\mathcal{M}_{\nu,\theta}({\eta}) }{\partial \mathcal{T}_{\nu,\theta\eta}}d{\eta}
&\leq
C_{T,f_0,\delta,\theta,q}\int_{0}^{1} \big(1+|v-U_{\eta}|^2\big)e^{-C_{T,f_0,\delta,\theta,q}(|v-U_{\eta}|^2+I^{2/\delta})}d{\eta} \cr
&\leq
C_{T,f_0,\delta,\theta,q}\int_{0}^{1}e^{-C_{T,f_0,\delta,\theta,q}(|v-U_{\eta}|^2+I^{2/\delta})}d{\eta}\cr
&\leq
C_{T,f_0,\delta,\theta,q}\int_{0}^{1}e^{-C_{T,f_0,\delta,\theta,q}(|v|^2+I^{2/\delta})}d{\eta}\cr
&\leq
C_{T,f_0,\delta,\theta,q}e^{-C_{T,f_0,\delta,\theta,q}(|v|^2+I^{2/\delta})}.
\end{align*}
In the last line, we have used
\begin{align*}
|U_{\eta}|\leq \eta|U_f|+(1-\eta)|U_g|\leq C_{T,f_0,\delta,q}
\end{align*}
which holds when $f,g\in \Omega$.\newline

Finally, we  turn to the proof of the proposition, which is almost done. Plugging the above inequalities into (\ref{Msplit}), we have
\begin{align*}
&|\mathcal{M}_{\nu,\theta}(f)-\mathcal{M}_{\nu,\theta}(g)|\cr
&\leq C\Big\{|\rho_f-\rho_g|+|U_f-U_g|+|\mathcal{T}_{\nu,\theta f}-\mathcal{T}_{\nu,\theta g}|+|T_{I,\delta f}-T_{I,\delta g}|\Big\}e^{-C(|v|^2+I^{2/\delta})} \cr
&\leq C\|f-g\|_{L_q^{\infty}}e^{-C(|v|^2+I^{2/\delta})}.
\end{align*}
Multiplying $(1+|v|^2+I^{\frac{2}{\delta}})^\frac{q}{2}$ and taking supremum  on both sides, we get the desired estimate.
\end{proof}
\section{proof of the main theorem}
In the mild form, (\ref{mildform}) reads
\begin{align*}
f^{n+1}(x,v,t,I)&= e^{-A_{\nu,\theta}t}f_0(x-vt,v,I) +A_{\nu,\theta}\int_{0}^{t} e^{-A_{\nu,\theta}(t-s)} \mathcal{M}_{\nu,\theta}(f^{n})(x-(t-s)v,v,s,I)ds, \cr
f^{n}(x,v,t,I)&= e^{-A_{\nu,\theta}t}f_0(x-vt,v,I) +A_{\nu,\theta}\int_{0}^{t} e^{-A_{\nu,\theta}(t-s)} \mathcal{M}_{\nu,\theta}(f^{n-1})(x-(t-s)v,v,s,I)ds.
\end{align*}
Taking difference and applying Proposition \ref{Lip_int}, we obtain
\begin{align*}
\|f^{n+1}(t)-f^{n}(t)\|_{L_q^{\infty}} &\leq
A_{\nu,\theta}\int_{0}^{t} e^{-A_{\nu,\theta}(t-s)}\|\mathcal{M}_{\nu,\theta}(f^{n}(t))-\mathcal{M}_{\nu,\theta}(f^{n-1}(t))\|_{L_q^{\infty}}ds \cr
&\leq A_{\nu,\theta}C_{Lip} \int_{0}^{t} \|f^{n}(t)-f^{n-1}(t)\|_{L_q^{\infty}} ds.
\end{align*}
Iterating this inequality,
\begin{align*}
\|f^{n+1}(t)-f^{n}(t)\|_{L_q^{\infty}} &\leq
A_{\nu,\theta}^nC_{Lip}^n \int_{0}^{t}\int_{0}^{s_1}\cdots\int_{0}^{s_{n-1}}\|f^1(s_n)-f^0(s_n)\|_{L_q^{\infty}}ds_n\cdots ds_2ds_1 \cr
&\leq
A_{\nu,\theta}^nC_{Lip}^n \int_{0}^{t}\int_{0}^{s_1}\cdots\int_{0}^{s_{n-1}}e^{-A_{\nu,\theta}s_n}\|f_0\|_{L_q^{\infty}}ds_n\cdots ds_2ds_1 \cr
&\leq
A_{\nu,\theta}^nC_{Lip}^n \frac{t^n}{n!}\|f_0\|_{L_q^{\infty}}.
\end{align*}
This immediately gives for $n > m$
\begin{align*}
\sup_{0\leq t\leq T}\|f^n(t)-f^{m}(t)\|_{L_q^{\infty}}\leq \left(e^{A_{\nu,\theta}C_{Lip}T}-\sum_{k=0}^{m-1}\frac{\big(A_{\nu,\theta}C_{Lip}T\big)^k}{k!}\right)\|f_0\|_{L_q^{\infty}}.
\end{align*}
Therefore, we conclude that $\{f^n\}$ is a Cauchy sequence and converges to an element, say $f$, in $\Omega$. It is standard to check that $f$ is the mild solution:
\begin{align*}
f(t,x,v,I)=e^{-A_{\nu,\theta}t}f_0(x-vt,v,I)+A_{\nu,\theta}\int_{0}^{t}e^{-A_{\nu,\theta}(t-s)}\mathcal{M}_{\nu,\theta}(f)(s,x-(t-s)v,v,I)ds.
\end{align*}
This proves the existence and estimates $(1)$ and $(2)$.\newline
For the proof of conservation laws, we find that Proposition (\ref{MconF_theta}) and the  Lebesgue differentiation theorem give from the above mild form,
\begin{align*}
\frac{d}{dt}f(t,x+tv,v,I)=A_{\nu,\theta}\{\mathcal{M}_{\nu,\theta}(f)-f\}(t,x+tv,v,I)
\end{align*}
for almost all $t$. Multiplying $1,v,\frac{1}{2}|v|^2+I^{\frac{2}{\delta}}$ on both sides and integrating with respect to $x,v, I$, we obtain $(3)$.\newline
The entropy dissipation estimate in the form of $(4)$ was established in \cite{ALPP,BS,PY2} at the formal level.
But the lower and upper bounds for $f$ and the macroscopic fields justify all those formal
computations given in \cite{ALPP,BS,PY2}. This completes the proof.

%
%
%
%
%
\section{When $\theta=0$}
Recall that the l.h.s of equivalence estimates in Lemma \ref{Temperature} vanish, and the r.h.s of the inequality in Proposition \ref{MconF_theta} blows up, as $\theta$ tends to zero. Therefore, the argument we've developed so far does not work for $\theta=0$.
We, however, observe that the polyatomic Gaussian $\mathcal{M}_{\nu,\theta}(f)$ in this case is completely split into the  translational internal energy part and the non-translational internal energy part as follows:
\begin{align*}
\mathcal{M}_{\nu,0}(f)&=\frac{\rho\Lambda_{\delta}}{\sqrt{\det(2\pi\mathcal{T}_{\nu,0})}T_{I,\delta}^{\frac{\delta}{2}}}e^{-\frac{1}{2}(v-U)^{\top}\mathcal{T}_{\nu,0}^{-1}(v-U)-\frac{I^{\frac{2}{\delta}}}{T_{I,\delta}}}\cr
&=\left(\frac{\rho}{\sqrt{\det(2\pi\mathcal{T}_{\nu,0})}}e^{-\frac{1}{2}(v-U)^{\top}\mathcal{T}_{\nu,0}^{-1}(v-U)}\right)\Bigg(\frac{\Lambda_{\delta}}{T_{I,\delta}^{\delta/2}}e^{-\frac{I^{2/\delta}}{T_{I,\delta}}}\Bigg)
\end{align*}
in the sense that $\mathcal{T}_{\nu,0}$ and $T_{I,\delta}$ does not share the common factor $T_{\delta}$. Now,
if we define
\begin{align*}
g(t,x,v)=\int_{\mathbb{R}^+}f(t,x,v,I)dI
\end{align*}
and integrate $(\ref{ESBGK})$ with respect to $I$, we get
\begin{align}\label{monatomic}
\begin{split}
\partial_tg+v\cdot\nabla_xg&=A_0\left\{\mathcal{M}_{\nu}(g)-g\right\}\cr
g_0(x,v)&=\int_{\mathbb{R}_+}f_0(x,v,I)dI,
\end{split}
\end{align}
where
\[
\mathcal{M}_{\nu}(g)\equiv\int_{\mathbb{R}_+}\mathcal{M}_{\nu,0}(f)dI=\frac{\rho}{\sqrt{\det(2\pi\mathcal{T}_{\nu})}}e^{-\frac{1}{2}(v-U)^{\top}\mathcal{T}_{\nu}^{-1}(v-U)}
\]
and
\[
\mathcal{T}_{\nu}\equiv \mathcal{T}_{\nu,0}=(1-\nu)T_{tr}+\nu\Theta.
\]
Note that $\rho$, $U$, $\mathcal{T}_{\nu}$ are naturally interpreted as macroscopic fields of $g$, and hence, so is $\mathcal{M}_{\nu}(g)$.
This is exactly the ellipsoidal BGK model for monatomic particles  \cite{ALPP,BS,Holway}. Relevant existence result for (\ref{monatomic}) can be found in \cite{PY1,Yun2, Yun3}.
Thus, in the case that $\theta=0$, our problem (\ref{ESBGK}) should be understood in the form (\ref{monatomic}).
This dichotomy between the two case, $\theta=0$ and $0<\theta\leq1$, was also observed in \cite{PY2,Yun33}.
\section{Appendix: Conservation laws}
In this appendix, we prove the cancellation property (\ref{cancellation}) for reader's convenience.\newline
$\bullet$ Mass conservation: Note that
\begin{align}\label{1}
\begin{split}
&\int_{\mathbb{R}^3\times\mathbb{R}^+}\mathcal{M}_{\nu,\theta}(f)dvdI \cr
&\qquad=\int_{\mathbb{R}^3\times\mathbb{R}^+}\frac{\rho\Lambda_{\delta}}{
\sqrt{\det(2\pi \mathcal{T}_{\nu,\theta})}(T_\theta)^\frac{\delta}{2}}e^{-\frac{1}{2}(v-U)^{\top}\mathcal{T}_{\nu,\theta}^{-1}(v-U)-\frac{I^{2/\delta}}{T_{\theta}}}dvdI \cr
&\qquad=\rho\int_{\mathbb{R}^3}\frac{1}{\sqrt{\det(2\pi\mathcal{T}_{\nu,\theta})}}e^{-\frac{1}{2}(v-U)^{\top}\mathcal{T}_{\nu,\theta}^{-1}(v-U)}dv \int_{\mathbb{R}^+}\frac{\Lambda_{\delta}}{T_{\theta}^{\frac{\delta}{2}}}e^{-\frac{I^{2/\delta}}{T_{\theta}}}dI.
\end{split}
\end{align}
Make a change of variable $X=\frac{1}{\sqrt{2}}\mathcal{T}_{\nu,\theta}^{-\frac{1}{2}}(v-U)$, so that
\begin{align*}
dX&=(\sqrt{2})^{-3}\det\mathcal{T}_{\nu,\theta}^{-\frac{1}{2}}dv
=\frac{(\sqrt{\pi})^3dv}{\sqrt{\det(2\pi\mathcal{T}_{\nu,\theta})}},\cr
\frac{1}{2}(v-U)^{\top}\mathcal{T}^{-1}_{\nu,\theta}(v-U)&=\left\{\frac{1}{\sqrt{2}}\mathcal{T}^{-1/2}_{\nu,\theta}(v-U)\right\}^{\top}\left\{\frac{1}{\sqrt{2}}\mathcal{T}^{-1/2}_{\nu,\theta}(v-U)\right\}=X^{\top}X=|X|^2.
\end{align*}
Therefore,
\begin{align*}
\int_{\mathbb{R}^3}\frac{1}{\sqrt{\det(2\pi\mathcal{T}_{\nu,\theta})}}e^{-\frac{1}{2}(v-U)^{\top}\mathcal{T}_{\nu,\theta}^{-1}(v-U)}dv
&=\frac{1}{(\sqrt{\pi})^3}\int_{\mathbb{R}^3}e^{-|X|^2}dX=1.
\end{align*}
On the other hand, putting $I/(T_{\theta})^{\delta/2}=J$,  we get
\begin{align*}
\int_{\mathbb{R}^+}\frac{\Lambda_{\delta}}{T_{\theta}^{\frac{\delta}{2}}}e^{-\frac{I^{2/\delta}}{T_{\theta}}}dI
&=
\Lambda_{\delta}\int_{\mathbb{R}^+}e^{-J^{2/\delta}}dJ=1.
\end{align*}
Hence, we have from (\ref{1})
\[
\int_{\mathbb{R}^3\times\mathbb{R}^+}\mathcal{M}_{\nu,\theta}(f)dvdI =\rho=\int_{\mathbb{R}^3\times\mathbb{R}^+}fdvdI.
\]
\noindent$\bullet$ Momentum conservation: We write
\begin{align}\label{2}
\begin{split}
\int_{\mathbb{R}^3\times\mathbb{R}^+}v\mathcal{M}_{\nu,\theta}(f)dvdI
&=\int_{\mathbb{R}^3\times\mathbb{R}^+}(v-U+U)\mathcal{M}_{\nu,\theta}(f)dvdI \cr
&=\int_{\mathbb{R}^3\times\mathbb{R}^+}(v-U)\mathcal{M}_{\nu,\theta}(f)dvdI +\rho U,
\end{split}
\end{align}
and make the same change of variable: $X=\frac{1}{\sqrt{2}}\mathcal{T}_{\nu,\theta}^{-\frac{1}{2}}(v-U)$, $I/(T_{\theta})^{\delta/2}=J$ to get
\begin{align*}
&\int_{\mathbb{R}^3\times\mathbb{R}^+}(v-U)\mathcal{M}_{\nu,\theta}(f)dvdI \cr
&\quad=\rho\int_{\mathbb{R}^3}\frac{(v-U)}{\sqrt{\det(2\pi\mathcal{T}_{\nu,\theta})}}e^{-\frac{1}{2}(v-U)^{\top}\mathcal{T}_{\nu,\theta}^{-1}(v-U)}dv \int_{\mathbb{R}^+}\frac{\Lambda_{\delta}}{T_{\theta}^{\frac{\delta}{2}}}e^{-\frac{I^{2/\delta}}{T_{\theta}}}dI \cr
&\quad=
\frac{\sqrt{2}\rho \mathcal{T}^{1/2}_{\nu,\theta}}{(\sqrt{\pi})^3}
\int_{\mathbb{R}^3}Xe^{-|X|^2}dX \cr
&\quad=0.
\end{align*}
Therefore, (\ref{2}) yields
\[
\int_{\mathbb{R}^3\times\mathbb{R}^+}v\mathcal{M}_{\nu,\theta}(f)dvdI=\rho U=
\int_{\mathbb{R}^3\times\mathbb{R}^+}vfdvdI.
\]
\noindent$\bullet$ Energy conservation:
To compute the translational part, we again set $X=\frac{1}{\sqrt{2}}\mathcal{T}_{\nu,\theta}^{-\frac{1}{2}}(v-U)$. Then, since
\[
|v-U|^2=(v-U)^{\top}(v-U)=\big(\sqrt{2}\mathcal{T}^{1/2}_{\nu,\theta}X\big)^{\top}\big(\sqrt{2}\mathcal{T}^{1/2}_{\nu,\theta}X\big)
=2X^{\top}\mathcal{T}_{\nu,\theta}X,
\]
we have
\begin{align}\label{energy conservation2}
\begin{split}
\int_{\mathbb{R}^3\times\mathbb{R}^+}\frac{|v-U|^2}{2}
\mathcal{M}_{\nu,\theta}(f)dvdI
&= \frac{\rho}{(\sqrt{\pi})^3}
\int_{\mathbb{R}^3} \big\{X^{\top}\mathcal{T}_{\nu,\theta}X\big\} e^{-|X|^2}dX \cr
&=
\frac{\rho}{(\sqrt{\pi})^3}
\int_{\mathbb{R}^3} \bigg\{\sum_{ij}X_iX_j\mathcal{T}^{ij}_{\nu,\theta}\bigg\}e^{-|X|^2}dX \cr
&=
\frac{\rho}{(\sqrt{\pi})^3}
\int_{\mathbb{R}^3} \bigg\{\sum_{i=1}^3X_i^2\mathcal{T}^{ii}_{\nu,\theta}\bigg\}e^{-|X|^2}dX \cr
&=
\frac{\rho}{(\sqrt{\pi})^3}
\sum_{i=1,2,3}\mathcal{T}^{ii}_{\nu,\theta}\bigg(\int_{\mathbb{R}^3}X_i^2e^{-|X|^2}dX\bigg) \cr
&= \frac{\rho}{2}tr\mathcal{T}_{\nu,\theta} \cr
&=
\frac{3}{2}\rho\big\{(1-\theta)T_{tr}+\theta T_{\delta}\big\},
\end{split}
\end{align}
where we used
\[
\bigg(\int_{\mathbb{R}^3}X_i^2e^{-|X|^2}dX\bigg)=\frac{\sqrt{\pi^3}}{2}.
\]
For the non-translational part, one finds
\begin{align*}
&\int_{\mathbb{R}^3\times\mathbb{R}^+}I^{2/\delta}
\mathcal{M}_{\nu,\theta}(f)dvdI \cr
&\qquad=
\rho\Lambda_{\delta}\int_{\mathbb{R}^+}\frac{I^{2/\delta}}{T_{\theta}^{\delta/2}}e^{-\frac{I^{2/\delta}}{T_{\theta}}}dI\int_{\mathbb{R}^3}\frac{1}{\sqrt{\det(2\pi\mathcal{T}_{\nu,\theta})}}e^{-\frac{1}{2}(v-U)^{\top}\mathcal{T}_{\nu,\theta}^{-1}(v-U)}dv \cr
&\qquad=\rho\Lambda_{\delta}\int_{\mathbb{R}^+}\frac{I^{2/\delta}}{T_{\theta}^{\delta/2}}e^{-\frac{I^{2/\delta}}{T_{\theta}}}dI.
\end{align*}
Let $X=I^{2/\delta}/T_{\theta}$, then $dI=\frac{\delta}{2}T_{\theta}^{\delta/2}X^{\delta/2-1}dX$, and thus,
\begin{align}\label{energy conservation1}
\begin{split}
\int_{\mathbb{R}^3\times\mathbb{R}^+}I^{2/\delta}
\mathcal{M}_{\nu,\theta}(f)dvdI
&=\rho\Lambda_{\delta}\int_{\mathbb{R}^+}\frac{I^{2/\delta}}{T_{\theta}^{\delta/2}}e^{-\frac{I^{2/\delta}}{T_{\theta}}}dI \cr
&=\frac{\delta}{2}\rho\Lambda_{\delta}T_{\theta}\int_{\mathbb{R}^+}X^{\delta/2}e^{-X}dX \cr
&=\frac{\delta}{2}\rho T_{\theta},
\end{split}
\end{align}
where we have used
\begin{align*}
\frac{\delta/2\rho T_{\theta}\int_{\mathbb{R}^+}X^{\delta/2}e^{-X}dX}{\int_{\mathbb{R}^+}e^{-I^{2/\delta}}dI}
&=
\frac{\delta/2\rho T_{\theta}\int_{\mathbb{R}^+}X^{\delta/2}e^{-X}dX}{\delta/2\int_{\mathbb{R}^+}X^{\delta/2-1}e^{-X}dX} \cr
&=\frac{\rho T_{\theta}\left[\left(-X^{\delta/2}e^{-X}\right)_{X=0}^{X=\infty}+\delta/2\int_{\mathbb{R}^+}X^{\delta/2-1}e^{-X}dX \right] }{\int_{\mathbb{R}^+}X^{\delta/2-1}e^{-X}dX}\cr
&=\frac{\delta}{2}\rho T_{\theta}.
\end{align*}

Combining (\ref{energy conservation1}) with (\ref{energy conservation2}) and recalling the definition of $T_{\delta}$ in (\ref{convex}), we get
\begin{align*}
&\int_{\mathbb{R}^3\times\mathbb{R}^+}\left(\frac{|v-U|^2}{2}+I^{2/\delta}\right)
\mathcal{M}_{\nu,\theta}(f)dvdI \cr
&\qquad=\frac{3}{2}\rho\big\{(1-\theta)T_{tr}+\theta T_{\delta}\big\}
+\frac{\delta}{2}\rho \big\{(1-\theta)T_{I,\delta}+\theta T_{\delta}\big\}\cr
&\qquad=\frac{3+\delta}{2}(1-\theta)\rho\Big\{\frac{3}{3+\delta}T_{tr}+\frac{\delta}{3+\delta}\delta T_{I,\delta}\Big\}+\frac{3+\delta}{2}\rho\theta T_{\delta}\cr
&\qquad= \frac{3+\delta}{2}\rho T_{\delta}\cr
&\qquad=\int_{\mathbb{R}^3\times\mathbb{R}^+}\left(\frac{|v-U|^2}{2}+I^{2/\delta}\right)fdvdI.
\end{align*}
\noindent{\bf Acknowledgement}
This research was supported by Basic Science Research Program through the National Research Foundation of Korea(NRF) funded by the Ministry of Education(NRF-2016R1D1A1B03935955)
\bibliographystyle{amsplain}

\end{document}